\documentclass [a4paper, 12pt]{amsart}

\usepackage[margin=2.5cm]{geometry} 
\usepackage{amsmath}
\usepackage{amsfonts}
\usepackage{amssymb}
\usepackage{amsthm}
\usepackage{graphicx}
\usepackage{pb-diagram}
\usepackage{epstopdf}
\usepackage{fancyhdr}
\usepackage[all]{xy}

\usepackage{amsthm}
\usepackage{latexsym}
\usepackage{times}
\usepackage{graphicx}
\usepackage{yhmath}

\newtheorem{cor}{Corollary}[section]
\newtheorem{te}[cor]{Theorem}
\newtheorem{p}[cor]{Proposition}

\newtheorem{lemma}[cor]{Lemma}
\theoremstyle{definition}
\newtheorem{de}[cor]{Definition}
\theoremstyle{remark}
\newtheorem{ob}[cor]{Observation}
\newtheorem{ex}[cor]{Example}

\newtheorem{nt}[cor]{Notation}

\newcommand{\cz}{\mathbb{C}}
\newcommand{\nz}{\mathbb{N}}
\newcommand{\zz}{\mathbb{Z}}

\newcommand{\rz}{\mathbb{R}}

\newcommand{\ff}{\mathbb{F}}

\newcommand{\bb}{\mathcal{B}}
\newcommand{\gr}{\mathcal{G}}
\newcommand{\hh}{\mathcal{H}}
\newcommand{\tr}{\mathcal{T}}
\newcommand{\pp}{\mathcal{P}}

\newcommand{\lron}{\mathcal{L}}
\newcommand{\nr}{\mathcal{N}}

\newcommand{\vp}{\varphi}
\newcommand{\ve}{\varepsilon}
\newcommand{\es}{\emptyset}
\newcommand{\sm}{\setminus}

   {\begin{verse}%
     \small }%
   {\end{verse}}

\providecommand{\Z}{\mathbb{Z}} \providecommand{\N}{\mathbb{N}}

\DeclareMathOperator{\id}{id}

\def\tilde{\widetilde}

\def\PSL{\mathop{\rm PSL}\nolimits}

\def\PGL{\mathop{\rm PGL}}

\def\Z{\mathbb Z}

\def\cR{{\mathcal R}}

\def\N{\mathbb N}

\def\cR{{\mathcal R}}

\begin{document}

\title{A generalisation to Birkhoff - von Neumann theorem}

\author{Liviu P\u aunescu}
\address[L. P\u aunescu]{Institute of Mathematics of the Romanian Academy, 21 Calea Grivitei Street, 010702 Bucharest, Romania}
\email{liviu.paunescu@imar.ro}
\author{Florin R\u adulescu}
\address[F. R\u adulescu]{Institute of Mathematics of the Romanian Academy, 21 Calea Grivitei Street, 010702 Bucharest, Romania and University of Rome 2, Tor Vergata, Rome 00185 Italy, Visiting 02-08.2015, Dept of Mathematical Sciences,
Universitetspark 5
2100 Copenhagen  Denmark}
\email{radulesc@mat.uniroma2.it}
\thanks{Work supported by grant number PN-II-ID-PCE-2012-4-0201 of the Romanian Ministry of Education, CNCS-UEFISCDI. L.P. was partially supported by the Austria-Romania research cooperation grant GALS on Sofic groups. F.R was supported in part by PRIN-MIUR and GNAMPA-INdAM}

\begin{abstract}
The classic Birkhoff- von Neumann theorem states that the set of doubly stochastic matrices is the convex hull of the permutation matrices. In this paper, we study a generalisation of this theorem in the type $II_1$ setting. Namely, we replace a doubly stochastic matrix with a collection of measure preserving partial isomorphisms, of the unit interval, with similar properties. We show that a weaker version of this theorem still holds.
\end{abstract}
\maketitle

A matrix $a\in M_m(\rz_+)$ with positive real entries is \emph{doubly stochastic} if $\sum_ka(i,k)=1=\sum_ka(k,j)$ for every $i,j$, while a \emph{permutation matrix} is such a matrix with only one non-zero entry on each row and column.

The classic Birkhoff- von Neumann theorem can be restated as follows. Let $P_m\subset M_m(\rz_+)$ be the set of permutation matrices and define $B_m^n=\{a\in M_m(\nz):\sum_ka(i,k)=n=\sum_ka(k,j)\mbox{ for every }i,j\}$. Then for each $a\in B_m^n$ there exists $p\in P_m$ such that $p(i,j)\leqslant a(i,j)$, equivalently $a-p\in M_m(\nz)$. It follows, by an easy induction, that each $a\in B_m^n$ is a sum of $n$ permutation matrices.

Various possible generalisations of this theorem had been investigated over the years.
For the infinite discrete case, consider the functions $a:\nz\times\nz\to\rz^+$. The notions of \emph{doubly stochastic} and \emph{permutation matrices} can be defined in a similar way. The theorem still holds in this case, if we ask uniform convergence for the sums $\sum_ka(i,k)$ and $\sum_ka(k,j)$. Details can be checked in \cite{Is}.

The infinite continuum case is more challenging. Let $I=[0,1]$ endowed with the Lebesgue measure $\mu$. One defines a \emph{doubly stochastic measure} as a probability measure $\lambda$ on $I^2$ such that $\lambda(A\times I)=\lambda(I\times A)=\mu(A)$. The analogue of a permutation matrix is a measure concentrated on the graph of a measure preserving automorphism of $I$. There are extreme doubly stochastic measures that are not of this type. In \cite{Lo}, two such examples are constructed. The first one is supported on a union of two non-$\mu$ preserving automorphisms of $I$, while the second is not supported on a countable union of automorphisms. Other constructions of such extreme doubly stochastic measures can be found in \cite{Li} or \cite{Se-Sh}. The article \cite{Br} contains a functional analysis study of the set of doubly stochastic measures.

In this paper we deal only with doubly stochastic measures that are supported on a countable union of measure preserving automorphisms. We restrict our study to maps $a:E\to\nz$, where $E$ is a countable measure-preserving equivalence relation and $\sum_{zEx}a(x,z)=n=\sum_{zEy}a(z,y)$.

\section{Introduction}

\subsection{Preliminaries}

Let $(X,\mathcal{B},\mu)$ be a standard probability space, and $E\subset X^2$, a countable, measure preserving equivalence relation. Recall the \emph{full group} of $E$ and the \emph{pseudo-group of partial isomorphisms}:
\begin{align*}
[E]=&\{\theta:X\to X:graph(\theta)\subset E,\mbox{ measurable bijection}\}\\
[[E]]=&\{\vp:A\to B:A,B\subset X,\ graph(\vp)\subset E,\mbox{ measurable bijection}\}
\end{align*}

In this article, elements in $[E]$ should be considered \emph{generalised permutation matrices}, while elements in $[[E]]$ play the role of \emph{generalised one entries} for our doubly stochastic matrices. 

It is well-known that each such equivalence relation is a countable union of graphs of elements in $[[E]]$. 

\begin{te}(Theorem 18.10 of \cite{Ke})\label{Kechris}
Let $E$ be a countable, Borel equivalence relation. Then $E=\bigcup_{n\in\nz}F_n$, where $(F_n)_n$ are Borel graphs.
\end{te}

The \emph{counting measure} on $E$ is a useful tool in this paper. For a Borel subset $C\subset E$ define:
\[\nu(C)=\int_XCard\big(C\cap(\{x\}\times X)\big)d\mu(x)=\int_XCard\big(C\cap(X\times\{x\})\big)d\mu(x).\]

The last equality is due to the measure preserving property of $E$, and the terms in that equality are called the \emph{right} and \emph{left counting measures}.

We denote by $\chi(A)$ or $\chi_A$ the characteristic function of $A$, by $A^c$ the completent in $X$ of $A$, and $A\Delta B$ is the symmetric difference of sets $A$ and $B$. Also $flip:X^2\to X^2$ is defined as $flip(x,y)=(y,x)$. For $\vp\in[[E]]$ the set $graph(\vp)\subset E$ is defined as $\{(x,\vp(x)):x\in dom(\vp)\}$.

\subsection{Basic definitions} We now define the main object of study of this paper.

\begin{de}
A \emph{doubly stochastic element}, for short a DSE, \emph{of multiplicity n}, is a collection of elements in $[[E]]$, $\Phi=\{\vp_i:A_i\to B_i:i\}$, such that $\sum_i\chi_{A_i}=n\cdot Id$ and  $\sum_i\chi_{B_i}=n\cdot Id$.
\end{de}

Alternatively we can view a DSE as a function $f:E\to\nz$.

\begin{de}\label{associated matrix}
Let $\Phi=\{\vp_i:i\}$ be a DSE of multiplicity $n$. The \emph{associated matrix} $M(\Phi):E\to\nz$, is defined as $M(\Phi)=\sum_i\chi(graph(\vp_i))$.
\end{de}

The associated matrix $M(\Phi)$ has the property that $\sum_zM(\Phi)(x,z)=n=\sum_zM(\Phi)(z,y)$ for $\mu$-almost all $x$ and $y$. Also, using Theorem \ref{Kechris}, a function with these properties can be transformer into a DSE. All in all, a DSE and its associated function is the same information. 

Due to the definition of a DSE, $M(\Phi):E\to\nz$ is a finite function, in the sense of Feldman-Moore, \cite{Fe-Mo}, see Definition \ref{finite function} below.  A  DSE can be finite or countable, depending on the number of elements in $[[E]]$ that composes it. In general, by a DSE we mean a finite DSE. We now prove that this is not a relevant restriction.

\begin{de}
Two doubly stochastic elements $\Phi=\{\vp_i:i\}$ and $\Psi=\{\psi_j:j\}$ are called \emph{equivalent} if they have the same associated matrix, i.e. $\sum_i\chi(graph(\vp_i))=\sum_j\chi(graph(\psi_j))$.
\end{de}

\begin{p}
An infinite DSE is equivalent to a finite DSE.\label{equiv to finite}
\end{p}
\begin{proof}
Let $\Phi=\{\vp_i:A_i\to B_i:i\}$ be such that $\sum_i\chi_{A_i}=n\cdot Id$ and  $\sum_i\chi_{B_i}=n\cdot Id$ for some $n\in\nz^*$. It is easy to construct $\{\theta_j:X\to X:j=1,\ldots,n\}$ such that $\sum_{j=1}^n\chi_{graph(\theta_j)}=\sum_i\chi_{graph(\vp_i)}$. These maps $\theta$ need not be elements in $[E]$, i.e. they may not be injective. We show that each of these maps can be decomposed into $n$ elements in $[[E]]$.

Choose $``<"$ a Borel total ordering on $X$. Let $T_j=\theta_j(X)$. For each $x\in T_j$, $\theta_j^{-1}(x)$ is composed of at most $n$ elements. Define $S_j^1=\{min\{\theta_j^{-1}(x)\}:x\in T_j\}$ and note that $\psi_j^1:S_j^1\to T_j$, $\psi_j^1(x)=\theta_j(x)$ is an element of $[[E]]$.

Consider now $T_j^2=\{x\in T_j:Card(\theta_j^{-1}(x))\geqslant 2\}$ and $S_j^2=\{\mbox{second min}\{\theta_j^{-1}(x)\}:x\in T_j^2\}$. Construct $\psi_j^2:S_j^2\to T_j^2$, $\psi_j^2(x)=\theta_j(x)$. By induction we get $\psi_j^k\in[[E]]$, $k=1,\ldots,m$ such that $\chi_{graph(\theta_j)}=\sum_k\chi_{graph(\psi_j^k)}$.
\end{proof}

%In general we assume that the sets $graph(\vp_i)$ are disjoint, where $(\vp_i)_i$ are the maps composing a DSE. All results presented in this paper are valid without this assumption. In order to avoid this restriction one has to replace the sets $graph(\vp_i)\subset E$ with functions $f:E\to\nz$ (these functions would play the role of characteristic functions). We believe that working with subsets of $E$ provides extra intuition.

\subsection{Distance between DSE}

The distance between two doubly stochastic elements $\Phi=\{\vp_i:i\}$ and $\Psi=\{\psi_j:j\}$, of the same multiplicity, is defined as:
\[d(\Phi,\Psi)=\int_E|M(\Phi)-M(\Psi)|d\nu.\]

This definition is very intuitive, it measures how much the partial morphisms $\vp_i$ differ from the partial morphisms $\psi_j$. Note that if $f_+$ and $f_-$ are the positive and negative part of the function $M(\Phi)-M(\Psi)$ then: $\int_E f_+d\nu=\int_E f_-d\nu=\frac12d(\Phi,\Psi)$.

\subsection{Decomposable DSE}
A finite number of elements in $[E]$ generate a doubly stochastic element, the same way that a number of permutation matrices generate a doubly stochastic matrix. The main question of this paper is whenever every DSE is obtained in this way.

\begin{de}
A doubly stochastic element $\Phi=\{\vp_i:A_i\to B_i:i=1,\ldots,m\}$ is \emph{decomposable} if there exists $\{\theta_j:X\to X:j=1,\ldots,n\}$ elements in $[E]$ such that $\sum_{j=1}^n\chi(graph(\theta_j))=\sum_{i=1}^m\chi(graph(\vp_i))$, i.e. $\Phi$ is equivalent to a DSE composed only of elements in $[E]$.
\end{de}

We shall show that for $n=2$ there exists a DSE that is not decomposable. The construction is based on the fact that not every Borel forest of lines is obtained from an element of $[E]$. However we do have a positive result if we replace ``decomposable" by the following weaker requirement.

\begin{de}
A doubly stochastic element $\Phi$ is \emph{almost decomposable} if for any $\ve>0$ there exists 
a decomposable DSE $\Psi$ such that $d(\Phi,\Psi)<\ve$.
\end{de}

We shall prove that all doubly stochastic elements are almost decomposable, i.e. the set of decomposable doubly stochastic elements is dense in the set of doubly stochastic elements endowed with the distance $d$.

\section{Doubly stochastic elements of multiplicity 2}

\subsection{Borel forest of lines}

In this section $F\subset X^2$ is an arbitrary hyperfinite aperiodic equivalence relation (no finite equivalence classes), and our doubly stochastic elements have multiplicity 2.

\begin{de}
A \emph{Borel forest of lines} $\lron$ for $F$ is an arrangement of the classes in $F$ like $\zz$-orderings with no distinguished direction. Formally $\lron\subset F$, such that $\lron=flip(\lron)$, $\lron$ generates $F$, and $Card(\lron\cap(\{x\}\times X))=2$ for $\mu$-almost all $x$.
\end{de}

\begin{ob}
Note that ``$F$ aperiodic" and ``$\lron$ generates $F$" imply that $\lron$ doesn't have cycles, so indeed classes of $F$ are arranged in an $\zz$-chain. 
\end{ob}

\begin{de}
For an element $\theta\in [F]$, that generates an aperiodic equivalence relation, define $\lron(\theta)=graph(\theta)\cup graph(\theta^{-1})$, the associated Borel forest of lines. If an arbitrary Borel forest of lines $\lron$ can be obtained by this construction we say that $\lron$ is \emph{generated by an automorphism}.
\end{de}

Remark 6.8 on page 21 of of \cite{Ke-Mi} shows that not every Borel forest of lines is generated by an automorphism. We present here a simplified version of that example.

\begin{ex}\label{Borel forest of lines} We take $X=[0,1]$ endowed with the Lebesgue measure.
Let $\vp_1,\vp_2:[0,1]\to[0,1]$ defined by $\vp_1(x)=1-x$ and $\vp_2(x)=\frac3{2^{n+1}}-x$ for $x\in(\frac1{2^{n+1}},\frac1{2^n})$. Note that $\vp_1^2=\vp_2^2=Id$. It follows that $\lron=graph(\vp_1)\cup graph(\vp_2)$ is a Borel forest of lines. Geometrically, $\vp_1$ is flipping the interval $[0,1]$, while $\vp_2$ is flipping the second half of $[0,1]$, the second half of $[0,1/2]$ and so on. The key observation is that $\vp_2\circ\vp_1$ is the odometer action, that is ergodic on $[0,1]$.

Assume that there exists $\theta:[0,1]\to[0,1]$ such that $\lron=\lron(\theta)$ almost everywhere. Let  $A=\{x\in[0,1]:\theta(x)=\vp_1(x)\}$. Let now $x\in A$. Then $\theta(x)=\vp_1(x)=1-x$, hence $\theta(1-x)$ must be equal to $\vp_2(1-x)$. Using the same argument in follows that $\vp_2(1-x)\in A$, so $A$ is invariant to the odometer action.

We have $x\in A$ if and only if $(1-x)\notin A$. Then $A\sqcup (1-A)=[0,1]$, so $\mu(A)=1/2$. Then $A$ is a set of measure $1/2$, invariant to an ergodic action. This is a contradiction.
\end{ex}

We can still save something out of this result if we ask that $\lron$ is generated on a set of  measure $1-\ve$, for any $\ve>0$.

\begin{te}\label{almost generated}
Let $\lron$ be a Borel forest of lines and $\ve>0$. Then there exists $\theta\in [F]$ such that $\nu\big(\lron\Delta\lron(\theta)\big)<\ve$.
\end{te}
\begin{proof}
The proof is an adaptation of the proof of Lemma 21.2 form \cite{Ke-Mi}, a result due independently to Gaboriau, and Jackson-Kechris-Louveau. We reproduce here parts of that proof for the reader's convenience.

Fix a sequence $\{g_i\}\subset[[F]]$ such that $\lron=\bigsqcup_i graph(g_i)$. Let $S\subset X$ be a Borel complete selection for $F$ such that $\mu(S)<\ve$ (see Lemma 6.7 from \cite{Ke-Mi}). For $x\in X\setminus S$, we define $\theta(x)$ to be the $\lron$-neighbour of $x$ that is closer to $S$. In case of equality, we use the smallest $g_i$. Formally, let $n$ be the least length of an $\lron$-path $x,x_1',\ldots,x_n'=z\in S$ from $x$ to $S$. Among all such paths, choose the ``lexicographically least one", using the maps $g_i$. We call this path $x,x_1,\ldots,x_n$ the canonical $\lron$-path from $x$ to $S$. Notice that, in this case $x_1,x_2,\ldots,x_n$ is the canonical path from $x_1$ to $S$.

For $x\in X\setminus S$ define $\theta_0(x)=x_1$. Then $\theta_0\in[[F]]$ and $graph(\theta_0)\subset\lron$. Moreover $\nu(graph(\theta_0))=1-\mu(S)>1-\ve$. Extend $\theta_0$ to an element $\theta\in[F]$. It follows that $\nu\big(\lron\Delta\lron(\theta)\big)<2\nu\big(\lron\setminus\lron(\theta_0)\big)<4\ve$ and we are done.
\end{proof}

\subsection{A counter-example for multiplicity two}

The Borel forest of lines constructed in Example \ref{Borel forest of lines} can be transferred into a symmetric DSE of multiplicity 2. This object is not decomposable as a symmetric DSE, but it is decomposable as a DSE (maps $\vp_1$ and $\vp_2$ constructed in the cited example provide a decomposition). More on symmetric DS elements in Section \ref{symm DSE}. In order to construct a indecomposable DSE of multiplicity 2 we perform the following construction.

\begin{p}
Let $\lron$ be a Borel forest of lines. Then we can construct $\{\vp_i:A_i\to B_i:i=1,\ldots,m\}$, a DSE of multiplicity 2, such that $\lron=\bigsqcup_{i\neq j}graph(\vp_i^{-1}\vp_j)$.
\end{p}
\begin{proof}
The idea is to construct a standard probability space $Y$ such that edges in $\lron$ first go to this space $Y$ and then return to $X$. There exists a collection of elements $\{\psi_k:S_k\to T_k\}_k\subset [[E]]$ such that $\lron=\bigsqcup_k(graph(\psi_k)\sqcup graph(\psi_k^{-1}))$. Now:
\[2=\nu(\lron)=2\sum_k\nu(\psi_k)=2\sum_k\mu(T_k),\]
so $\sum_k\mu(T_k)=1$. Let $(T_k^Y,\mu|_{T_k})$ be a copy of the set $T_k$ and define $Y=\bigsqcup_kT_k^Y$. Then $Y$ is a standard probability space. Now construct:
\begin{align*}
&\vp_{k,1}:S_k\to T_k^Y,\ \ \ \vp_{k,1}(x)=\psi_k(x);\\
&\vp_{k,2}:T_k\to T_k^Y,\ \ \ \vp_{k,2}(x)=x.
\end{align*}

Then $\vp_{k,2}^{-1}\vp_{k,1}(x)=\psi_k(x)$ and $\vp_{k,1}^{-1}\vp_{k,2}(x)=\psi_k^{-1}(x)$. The sets $\{T_k\}$ may not be disjoint, but their copies $\{T_k^Y\}$ are disjoint in $Y$. So $\vp_{k_1,s_1}^{-1}\vp_{k_2,s_2}$ is nonempty if and only if $k_1=k_2$. It follows that $\bigcup_{(k_1,s_1)\neq (k_2,s_2)}graph(\vp_{(k_1,s_1)}^{-1}\vp_{(k_2,s_2)})=\bigcup_k(graph(\psi_k)\cup graph(\psi_k^{-1}))=\lron$.

We want to prove that the collection $\{\vp_{k,s}\}$ is a DSE, as soon as we fix an isomorphism between $X$ and $Y$. An element $y\in T_k^Y$ is only in the images of the maps $\vp_{k,1}$ and $\vp_{k,2}$, so we are done with these elements. Any $x\in X$ belongs to exactly two sets selected from the collection $\{S_k,T_k\}_k$. Then $x$ is in the domain of exactly two maps from the set $\{\vp_{k,1},\vp_{k,2}\}_k$. It follows that $\{\vp_{k,s}\}$ is indeed a DSE. Use Proposition \ref{equiv to finite} to replace it by a finite DSE, if needed.
\end{proof}

\begin{te}
Let $\lron$ be a Borel forest of lines and construct $\Phi=\{\vp_i:A_i\to B_i:i=1,\ldots,m\}$ a DSE such that $\lron=\bigsqcup_{i\neq j}graph(\vp_i^{-1}\vp_j)$. If $\Phi$ is decomposable then $\lron$ is generated by an automorphism of $(X,\mu)$.
\end{te}
\begin{proof}
Let $\{\theta_j:X\to X:j=1,2\}$ be elements in $[E]$ such that $\sum_{j=1}^2\chi_{graph(\theta_j)}=\sum_{i=1}^m\chi_{graph(\vp_i)}$. Then $\lron=\bigcup_{i\neq j}graph(\vp_i^{-1}\vp_j)=graph(\theta_2^{-1}\theta_1)\cup graph(\theta_1^{-1}\theta_2)=\lron(\theta_2^{-1}\theta_1)$.
\end{proof}

\begin{ex}\label{counterexample}
Inspecting the proof and Example \ref{Borel forest of lines} we can actually come up with an indecomposable DSE. It is composed of the following partial isomorphisms:
\begin{align*} 
&\vp_1:(0,\frac12)\to(\frac12,1), &\vp_1(x)&=x+\frac12;\\
&\vp_2:(\frac12,1)\to(\frac12,1), &\vp_2(x)&=x;\\
&\psi_n^1:(\frac1{2^{n+1}},\frac3{2^{n+2}})\to(\frac1{2^{n+2}},\frac1{2^{n+1}}), &\psi_n^1(x)&=x-\frac1{2^{n+2}};\\
&\psi_n^2:(\frac3{2^{n+2}},\frac1{2^n})\to(\frac1{2^{n+2}},\frac1{2^{n+1}}),& \psi_n^2(x)&=x-\frac1{2^{n+1}}.
\end{align*}
Maps $\{\psi_n^i\}_{n\in\nz}$, for $i=1,2$ can be glued to one map onto $(0,\frac12)$.
\end{ex}

\subsection{Main result for multiplicity two} We shall prove this result for any multiplicity, but in case $n=2$ we have an easier proof based on the properties of Borel forests of lines.

\begin{te}
Any DSE of multiplicity 2 is almost decomposable.
\end{te}
\begin{proof}
Let $\Phi=\{\vp_i:A_i\to B_i:i=1,\ldots,m\}$ be a DSE of multiplicity 2. We construct a Borel forest of lines as follows. Let $Y=X\times\{1,2\}$ endowed with the product measure of $\mu$ and $Card$, so that the total measure of $Y$ is 2. Define $\lron_0=\bigcup_i\big((x,1),(\vp_i(x,2)\big)\subset Y^2$. Then $\lron=\lron_0\cup flip(\lron_0)$ is a Borel forest of lines on $Y$.

Now choose $\ve>0$. By Theorem \ref{almost generated} there exists $\theta:Y\to Y$ such that $m(\lron\Delta\lron(\theta))<2\ve$. Construct $\psi_1,\psi_2:X\to X$ defined by the equations $\theta(x,1)=(\psi_1(x),2)$ and $\theta(x,2)=(\psi_2^{-1}(x),1)$.

Let $\tr=\{\big((x,1),(\psi_j(x),2)\big):x\in X, j=1,2\}$, so that $\lron(\theta)=\tr\cup flip(\tr)$. It is easy to see that $\nu(\lron_0\Delta\tr)=\frac12\nu(\lron\Delta\lron(\theta))<\ve$. The conclusion follows as $\nu\big((\sum_{i=1}^m\chi_{graph(\vp_i)})\Delta(\sum_{j=1}^2\chi_{graph(\psi_j)})\big)=\nu(\lron_0\Delta\tr)$.
\end{proof}

\section{Main result}

In this section we prove that any DSE is almost decomposable. First, we give some definitions. For this section fix $\Phi=\{\xi_i\}_{i=1}^m$ a DSE of multiplicity $n$. Define $\hh=\bigcup_{i=1}^mgraph(\xi_i)\subset X\times X$. This set is the support of the associated matrix of $\Phi$.

\begin{nt}
For $C\subset X$ define the \emph{neighboring set} by $N(C)=\bigcup_{i=1}^m\xi_i(A_i\cap C)$.
\end{nt}

We shall obtain many useful inequalities using the equality of the right and left counting measures. Here is a first example that we prove in detail.

\begin{lemma}
For any $C\subset X$ we have $\mu(N(C))\geqslant\mu(C)$.
\end{lemma}
\begin{proof}
As $E$ is $\mu$-preserving, it follows that:
\[\int_Efd\nu=\int_{x\in X}\sum_{y\sim x}f(x,y)d\mu(x)=\int_{y\in X}\sum_{x\sim y}f(x,y)d\mu(y),\]
for each measurable function $f:E\to\cz$. Let $f(x,y)=\big(\sum_{i=1}^m\chi_{graph\vp_i}(x,y)\big)\times\chi_C(x)$. We can see that $\sum_y\chi_{graph(\vp_i)}(x,y)=\chi_{A_i}(x)$ and $\sum_x\big(\chi_{graph(\vp_i)}(x,y)\times\chi_C(x)\big)=\chi_{\vp_i(A_i\cap C)}(y)$. Then:
\[\int_{x\in X}\sum_{y\sim x}f(x,y)d\mu(x)=\int_{x\in C}\sum_{i=1}^m\sum_{y\sim x}\chi_{graph\vp_i}(x,y)d\mu(x)=\int_{x\in C}\sum_{i=1}^m\chi_{A_i}(x)=n\cdot\mu(C);\]
\[\int_{y\in X}\sum_{x\sim y}f(x,y)d\mu(y)=\int_{y\in X}\sum_{i=1}^m\big(\sum_y\chi_{graph(\vp_i)}(x,y)\times\chi_C(x)\big)d\mu(y)=\int_{y\in X}\sum_{i=1}^m\chi_{\vp_i(A_i\cap C)}(y)d\mu(y).\]

If $y\notin N(C)=\bigcup_{i=1}^m\vp_i(A_i\cap C)$ then $\sum_{i=1}^m\chi_{\vp_i(A_i\cap C)}(y)=0$. Independently of $y$, $\sum_{i=1}^m\chi_{\vp_i(A_i\cap C)}(y)\leqslant\sum_{i=1}^m\chi_{B_i}(y)=n$. All in all:
\[n\cdot\mu(C)=\int_{y\in X}\sum_{i=1}^m\chi_{\vp_i(A_i\cap C)}(y)d\mu(y)\leqslant\int_{y\in N(C)}n\cdot d\mu(y)=n\cdot\mu(N(C)).\]
\end{proof}

\begin{de}
A partial isomorphism $\theta\in[[E]]$ is called a \emph{piece} if $graph(\theta)\subset \hh$.
A piece $\theta:A\to B$ is called \emph{maximal} if $N(A^c)\subset B$.
\end{de}

A piece is maximal if there is no immediate way of extending it in a classical meaning. The next definition provides a notion of extension that is more suited to our context.

\begin{de}
Let $\theta:A\to B$ be a piece. An \emph{extension} of $\theta$ is a collection of pieces $\vp_i:S_{i-1}\to T_{i}$, $i=1,\ldots,k+1$ such that $S_i$ are disjoint for $i=0,1,\ldots,k$ and:
\begin{enumerate}
\item $S_0\subset A^c$ and $S_i\subset A\mbox{ for }i=1,\ldots,k$;
\item $T_i\subset B\mbox{ for }i=1,\ldots,k$ and $T_{k+1}\subset B^c$.
\item $\theta^{-1}(T_i)=S_i$ for $i=1,\ldots,k$;
\end{enumerate}
The number $k\in\nz$ is called the \emph{depth} of the extension.
\end{de}

Observe that a piece is maximal if and only if it admits no 0-depth extension. The name \emph{extension} is not arbitrary. From the information in the last definition one can construct a piece
$\theta':A\cup S_0\to B\cup T_{k+1}$ using $\theta$ and $\vp_i$, $i=0,1,\ldots,k$.
%\begin{de}
%A partial isomorphism $\theta\in[[E]]$ is called a \emph{piece} if $graph(\theta)\subset M$. A piece $\theta:A\to B$ is called \emph{extendable} if there exists another piece $\theta_1:A_1\to B_1$ such that $A\subset A_1$ and $\mu(A)<\mu(A_1)$. Note that we do not ask that $\theta_1|_A=\theta$.
%\end{de}

It can be proven that each piece, that is not defined on the whole space $X$, admits an extension. However this result is not sufficient to prove that there exists pieces defined on arbitrarily large sets. We need to control the size of these extensions. By a careful study of the problem, one understands that controlling the size requires also controlling the depth of the extension. Proposition \ref{prop:large_extension} provides a construction of an extension by controlling its size and depth. First we need two helpful lemmas. 

If $A,B\subset X$ are such that $\mu(A)>\mu(B)$ it is easy to see that there exist a piece from some $S\subset A$ to $T\subset B$ such that $\mu(S)\geqslant(\mu(A)-\mu(B))/2$. However we need the following version of this observation.

\begin{lemma}\label{main_lemma}
Let $A,B\subset X$ and let $\theta:C\to D$ be a piece such that $A\cap C=\emptyset$ and $B\cap D=\emptyset$. Then there exists $\theta_1:S\to T$ a piece with $S\subset A$, $T\subset(B\cup D)^c$ and:
\[\mu(S)\geqslant\frac{\mu(A)-\mu(B)}2-\frac{n-1}{2n}\mu(C).\]
\end{lemma}
\begin{proof}
Let $\theta_1:S\to T$ be a maximal piece with those properties. Then, if it can't be extended, it follows that $N(A\setminus S)\subset B\cup D\cup T$. By considering the left and right counting measure of $[(A\setminus S)\times X]\cap M$ we get:
\[n\mu(A\sm S)\leqslant n\mu(B)+(n-1)\mu(D)+n\mu(T).\]
As $\mu(S)=\mu(T)$ the conclusion follows.
\end{proof}

The next lemma is used to construct extensions of a maximal given depth.

\begin{lemma}
Let $\theta:A\to B$ be a piece and let $\psi_i:V_{i-1}\to W_i$ $i=1,\ldots,j+1$ be pieces such that $W_1,\ldots,W_j, W_{j+1}$ are disjoint subsets,  $W_i\subset B$ for $i\leqslant j$ and $V_i\subset\theta^{-1}(W_1\cup\ldots\cup W_i)$ for any $0<i\leqslant j$. Assume that $V_0\subset A^c$ and $W_{j+1}\not\subset B$. Then $\theta$ admits an extension of depth smaller or equal to $j$.
\end{lemma}
\begin{proof}
Let $T_1=W_{j+1}\cap B^c$. By hypothesis $\mu(T_1)>0$. Then $\psi_{j+1}^{-1}(T_1)\subset V_j\subset\theta^{-1}(W_1\cup\ldots\cup W_j)=\theta^{-1}(W_1)\cup\ldots\cup\theta^{-1}(W_j)$. It follows that there exists $i_1<j+1$ such that $\mu(\psi_{j+1}^{-1}(T_1)\cap\theta^{-1}(W_{i_1}))>0$.

Let $S_1=\psi_{j+1}^{-1}(T_1)\cap\theta^{-1}(W_{i_1})$ and $T_2=\theta(S_1)$. Then $T_2\subset W_{i_1}=\psi_{i_1}(V_{i_1-1})$. If $i_1=1$ then we are done, as $\psi_1$ restricted to $\psi_1^{-1}(T_2)$ and $\psi_{j+1}$ restricted to $S_1=\theta^{-1}(T_2)$ provide an extension of $\theta$ of depth $1$. If $i_1>1$ then $\psi_{i_1}^{-1}(T_2)\subset V_{i_1-1}\subset\theta^{-1}(W_1)\cup\ldots\cup\theta^{-1}(W_{i_1-1})$. Hence, there exists $i_2<i_1$ such that $\mu(S_2)>0$, where $S_2=\psi_{i_1}^{-1}(T_2)\cap\theta^{-1}(W_{i_2})$.

Inductively define $S_r=\psi_{i_{r-1}}^{-1}(T_r)\cap\theta^{-1}(W_{i_r})$ such that $\mu(S_r)>0$ and $T_{r+1}=\theta(S_r)$. If $i_r>1$ then $T_{r+1}\subset W_{i_r}=\psi_{i_r}(V_{i_r-1})$, so $\psi_{i_r}^{-1}(T_{r+1})\subset V_{i_r-1}\subset\theta^{-1}(W_1)\cup\ldots\cup\theta^{-1}(W_{i_r-1})$. Then there exists $i_{r+1}<i_r$ such that $S_{r+1}=\psi_{i_r}^{-1}(T_{r+1})\cap\theta^{-1}(W_{i_{r+1}})$ and $\mu(S_{r+1})>0$. 

If $i_r=1$ then $\psi_{i_r}^{-1}(T_{r+1})\subset V_0\subset A^c$ and this set $\psi_{i_r}^{-1}(T_{r+1})$ alternatively transported with maps $\psi_{i_r},\psi_{i_{r-1}},\ldots,\psi_{i_0}$ and $\theta^{-1}$ is an extension of $\theta$ of depth $r$ (where $i_0=j+1$). As $1=i_r<i_{r-1}<\ldots<i_1<i_0=j+1$ it follows that $r\leqslant j$.
\end{proof}

The following proposition is the main step in the proof of the result of this section.

\begin{p}\label{prop:large_extension}
Let $\theta:A\to B$ be a piece. Then there exists a piece $\theta_1:C\to D$ such that $\mu(C)\geqslant\mu(A)+(\frac{\mu(A^c)}{7n+\mu(A^c)})^2$.
\end{p}
\begin{proof}
Let $k=\lfloor\frac{7n}{\mu(A^c)}\rfloor$. We can assume that $\theta$ is maximal.

Let $\{\vp_i^j:S_i^j\to T_{i+1}^j|i=1,\ldots,k(j)\}_j$ be a maximal collection of extensions of $\theta$ of depth smaller or equal to $k$. This means that we require that the sets $(S_i^j)_{i,j}$ are disjoint and the sets $(T_i^j)_{i,j}$ are also disjoint. Additionally $k(j)\leqslant k$ for any $j$ and there is no extra extension $\{\vp_i\}_i$ of $\theta$ that can be added to the family without breaking at least one of these properties.

Let $S=\cup_jS_0^j$, $S_e=\cup_{i,j}S_i^j$ and $T_e=\cup_{i,j}T_i^j$. As the maximal depth of any extension is $k$, we have $\mu(S_e)\leqslant (k+1)\mu(S)$. Also $\mu(S_e)=\mu(T_e)$.

Using Lemma \ref{main_lemma} for $(A^c\sm S)$ and $T_e$, we deduce that there exists a piece $\psi_0:V_0\to W_1$ such that $V_0\subset (A^c\sm S)$, $W_1\subset X\sm T_e$ and 
\[\mu(W_1)\geqslant\frac{\mu(A^c\sm S)-\mu(T_e)}2.\]
As $\theta$ is maximal, it follows that $W_1\subset B$.

Now we apply Lemma \ref{main_lemma} to sets $\theta^{-1}(W_1)$, $T_e$ and the piece $\psi_0:V_0\to W_1$ to deduce the existence of $\psi_1:V_1\to W_2$ such that $V_1\subset\theta^{-1}(W_1)$, $W_2\subset X\sm(T_e\cup W_1)$ and:
\[\mu(W_2)\geqslant\frac{\mu(W_1)-\mu(T_e)}2-\frac{n-1}{2n}\mu(W_1)=\frac{\mu(W_1)}{2n}-\frac{\mu(T_e)}2.\]
As $\theta$ admits no new extension of depth less than $k$, it follows that $W_2\subset B$, so $W_2$ is actually a subset of $B\sm(T_e\cup W_1)$. 

For the next step, apply Lemma \ref{main_lemma} for $\theta^{-1}(W_1\cup W_2)$, $T_e\cup W_2$ and the piece $\psi_0:V_0\to W_1$. There exists $\psi_2:V_2\to W_3$ such that $V_2\subset\theta^{-1}(W_1\cup W_2)$, $W_3\subset X\sm(T_e\cup W_1\cup W_2)$ and:
\[\mu(W_3)\geqslant\frac{\mu(W_1\cup W_2)-\mu(T_e\cup W_2)}2-\frac{n-1}{2n}\mu(W_1)=\frac{\mu(W_1)}{2n}-\frac{\mu(T_e)}2.\]

If $W_3\not\subset B$ then, by the previous lemma, there exists an extension of $\theta$ of depth less than $2$. As this extension will use only maps $\psi_0,\psi_1$ and $\psi_2$ it will not intersect any extension from the family $\{\vp_i^j\}_j$. This contradicts the maximality of this family. It follows that $W_3\subset B$.

Inductively apply Lemma \ref{main_lemma} to $\theta^{-1}(W_1\cup\ldots\cup W_r)$, $T_e\cup W_2\cup\ldots\cup W_r$ and the piece $\psi_0:V_0\to W_1$ to get $\psi_r:V_r\to W_{r+1}$ such that $V_r\subset\theta^{-1}(W_1\cup\ldots\cup W_r)$, $W_{r+1}\subset X\sm(T_e\cup W_1\cup\ldots\cup W_r)$ and:
\[\mu(W_{r+1})\geqslant\frac{\mu(W_1\cup\ldots\cup W_r)-\mu(T_e\cup W_2\cup\ldots\cup W_r)}2-\frac{n-1}{2n}\mu(W_1)=\frac{\mu(W_1)}{2n}-\frac{\mu(T_e)}2.\]

As long as $r\leqslant k$, the previous lemma can be used to deduce that $W_{r+1}\subset B$.

In the end we have $k+1$ disjoint subsets of $B$. It follows that $\sum_{r=1}^{k+1}\mu(W_r)\leqslant\mu(B)$. Recall that $\mu(T_e)\leqslant(k+1)\mu(S)$. Using the above inequalities we get:
\begin{align*}
\mu(W_2) & \geqslant\frac{\mu(W_1)}{2n}-\frac{\mu(T_e)}2\geqslant\frac{\mu(A^c\sm S)-\mu(T_e)}{4n}-\frac{\mu(T_e)}2>\frac{\mu(A^c)-2(k+1)\mu(S)}{4n}-\frac{(k+1)\mu(S)}2\\ 
= & \frac{\mu(A^c)}{4n}-\frac{2(n+1)(k+1)\mu(S)}{4n}\\
\end{align*}
As $1>\mu(B)\geqslant\sum_{r=1}^{k+1}\mu(W_r)>(k+1)\mu(W_2)$ we get:
\begin{align*}
4n&>(k+1)(\mu(A^c)-2(n+1)(k+1)\mu(S))>7n-2(n+1)(k+1)^2\mu(S)\\
\mu(S)&>\frac{3n}{2(n+1)(k+1)^2}>\frac1{(k+1)^2}>\big(\frac{\mu(A^c)}{7n+\mu(A^c)}\big)^2.
\end{align*}
It follows that by using the extensions $\{\vp_i^j\}_j$ we can construct the required piece.
\end{proof}

\begin{te}\label{pass_to_limit}
For any DSE and any $\ve>0$, there exists a piece $\theta:A\to B$ such that $\mu(A)>1-\ve$.
\end{te}
\begin{proof}
Using the last proposition construct a sequence of pieces $\theta_i:A_i\to B_i$ such that $\mu(A_{i+1})>\mu(A_i)+(\frac{1-\mu(A_i)}{7n+1-\mu(A_i)})^2$. Then $(\mu(A_i))_i$ is a bounded, increasing sequence. Its limit $l$ must obey the inequality $l\geqslant l+(\frac{1-l}{7n+1-l})^2$. It follows that $0\geqslant (1-l)^2$ so $l=1$.
\end{proof}

\begin{te}\label{th:decDSE}
Any DSE is almost decomposable.
\end{te}
\begin{proof}
Obviously the proof goes by induction on the multiplicity of the DSE. Let $\Phi$ be a DSE of multiplicity $n$ and let $\ve>0$. Choose $\theta:A\to B$ a piece such that $\mu(A)>1-\ve/8$. We can assume that $\theta$ is maximal, i.e. there is no piece $\vp:C\to D$ with $C\subset A^c$ and $D\subset B^c$. Extend $\theta$ to $\tilde\theta\in[E]$. Our goal is to construct $\Psi$ a DSE of multiplicity $n-1$ such that $d(\Phi,\Psi\sqcup\{\tilde\theta\})$ is small.

Choose a sequence of pieces $\vp_i:C_i\to D_i$ such that $\bigsqcup_iC_i=A^c$ and a another sequence of pieces $\psi_i:V_i\to W_i$ such that $\bigsqcup_IW_i=B^c$. Then $\sum_i\mu(V_i)=\mu(B^c)=\mu(A^c)=\sum_i\mu(D_i)$. Finally choose a sequence of measure preserving partial morphisms $\delta_j:S_j\to T_j$ such that $\sum_j\chi_{S_j}=\sum_i\chi_{V_i}$ and $\sum_j\chi_{T_j}=\sum_i\chi_{D_i}$. Define $f=M(\Phi)-\chi(graph(\theta))-\sum_i\chi(graph(\vp_i))-\sum_i\chi(graph(\psi_i))+\sum_j\chi(graph(\delta_j))$. Then $\sum_xf(x,y)=n-1$ for any $y$ and $\sum_yf(x,y)=n-1$ for any $x$. By Theorem \ref{Kechris}, there exists $\Psi$, a DSE of multiplicity $n-1$, such that $M(\Psi)=f$.  Then:
\[d(\Phi,\Psi\sqcup\{\tilde\theta\})\leqslant\int\chi(graph(\tilde\theta))-\chi(graph(\theta))-\sum_i\chi(graph(\vp_i))-\sum_i\chi(graph(\psi_i))+\sum_j\chi(graph(\delta_j)).\]
Notice that $\int\sum_i\chi(graph(\vp_i))d\nu=\int\chi(graph(\psi_i))d\nu =\int\chi(graph(\delta_j)d\nu=\mu(A^c)$. It follows that $d(\Phi,\Psi\sqcup\{\tilde\theta\})\leqslant4\mu(A^c)<\ve/2$. 

By induction, there exists a decomposable DSE $\Psi_1$ such that $d(\Psi_1,\Psi)<\ve/2$. Then $\Psi_1\sqcup\{\tilde\theta\}$ is decomposable and $d(\Phi,\Psi_1\sqcup\{\tilde\theta\})<\ve$.
\end{proof}

\section{Symmetric doubly stochastic elements}\label{symm DSE}

A doubly stochastic element is \emph{symmetric} if, together with a morphism $\vp:A\to B$, it contains it's inverse $\vp^{-1}:B\to A$. In this section we show how to split a symmetric DSE of even multiplicity into a DSE and its inverse. First some formal definitions.

\begin{de}
For $\Phi=\{\vp_i:A_i\to B_i:i\}$ a DSE, define its inverse $\Phi^{-1}=\{\vp_i^{-1}:B_i\to A_i:i\}$. Note that $\Phi^{-1}$ is still a DSE of the same multiplicity as $\Phi$. 

A DSE $\Phi$ is called \emph{symmetric} if $\Phi$ is equivalent to $\Phi^{-1}$.
\end{de}

\begin{de}
For a DSE $\Phi$ define its symmetrisation $\mathcal{S}(\Phi)=\Phi\sqcup\Phi^{-1}$. Note that $\mathcal{S}(\Phi)$ is a symmetric DSE of multiplicity twice the multiplicity of $\Phi$.
\end{de}

The goal of this section is to prove that for any DSE of multiplicity $2n$, there exists a DSE of multiplicity $n$ such that its symmetrisation is arbitrarily close to the initial DSE. The proof is similar to the one in the last section, and it will closely follow the same sketch. However, there are some different points, and we have to readapt the lemmas that we used. 

For this section we now fix $\Phi=\{\xi_i\}_i$ a symmetric DSE of multiplicity $2n$. Let $\gr$ be its associated graph $\gr=\bigcup_igraph(\xi_i)$. In this section we assume that $(graph(\xi_i))_i$ are disjoint sets, so that $\nu(\gr)=2n$. This assumption doesn't change the proof in any way. It allows us to work with $\gr\subset E$, instead of working with $M(\Phi):E\to\nz$. We consider that working with a graph provides more intuition, while not simplifying the conceptual proof.

We note that $\gr$ is indeed a graph, i.e. $\gr=flip(\gr)$. Actually a $d$-regular (multi-) graph (every vertex has $d$ neighbours)  and a symmetric DSE of multiplicity $d$ is the exact same information. The first step of the proof is to ``divide" the graph $\gr$ into two disjoint parts.

\begin{de}
A \emph{division} of $\gr$ is a subset $\hh\subset\gr$ such that $\hh$ and $flip(\hh)$ are a partition of $\gr$.
\end{de}

This can be done by selecting a Borel order $``<"$ on $X$ and defining $\hh=\{(x,y):(x,y)\in\gr,x<y\}$, so divisions do exists. From this definition it follows that $\nu(\hh)=n$ for any division. Intuitively, a division puts a direction on each edge in $\gr$. A perfect division would be one for which the in-degree, is equal to the out-degree, is equal to $n$ for each vertex $x\in X$. We now formalise these observations.

Define $d_\hh:X\to\nz$ by $d_\hh(x)=Card\{y:(x,y)\in\hh\}$. By definition $\nu(\hh)=\int_Xd_\hh(x)d\mu(x)=n$. We define the \emph{error} of $\hh$ as $E(\hh)=\int_X|n-d_\hh(x)|d\mu(x)$. Our goal is to construct $\hh$ with arbitrarily small error. We do this by gradually improving the error of $\hh$.

Let $P_0=\{x:d(x)=n\}$, $P_-=\{x:d(x)<n\}$ and $P_+=\{x:d(x)>n\}$. Then $\{P_0,P_-,P_+\}$ is a partition of $X$ and it can be easily checked that $\mu(P_+)\leqslant E(\hh)\leqslant 2n\cdot\mu(P_+)$ and the same inequalities hold for $\mu(P_-)$. We now define the object we want to construct inside $\hh$.

\begin{de}
A \emph{better path} in $\hh$ is a collection of maps $\vp_i:V_{i-1}\to V_i$, $i=1,\ldots,k$ such that $graph(\vp_i)\subset\hh$. Additionaly $V_0,\ldots,V_k$ are disjoint subsets of $X$, $V_0\subset P_+$ and $V_k\subset P_-$. The number $k$ is called the \emph{length} of the path.
\end{de}

A better path is improving the error of $\hh$, by reversing the direction of the edges, as the next proposition shows.

\begin{p}
Let $\hh$ be a division of $\gr$ and let $\{\vp_i:i=1,\ldots k\}$ be a better path in $\hh$. Define $\mathcal{P}=\bigcup_{i=1}^kgraph(\vp_i)$. Then $\hh_1=\hh\sm\mathcal{P}\cup flip(\mathcal{P})$ is a division of $\gr$ and $E(\hh_1)=E(\hh)-2\mu(V_0)$.
\end{p}
\begin{proof}
As $\mathcal{P}\subset\hh$ it follows that $\mathcal{P}\cap flip(\mathcal{P})=\es$. Now we can see that $\hh_1$ and $flip(\hh_1)$ are a partition of $\gr$. Computing the degree, we get $d_{\hh_1}(x)=d_\hh(x)$ for $x\in X\sm(V_0\cup V_k)$, $d_{\hh_1}(x)=d_\hh(x)-1$ for $x\in V_0$ and $d_{\hh_1}(x)=d_\hh(x)+1$ for $x\in V_k$. As $V_0\subset P_+$ and $V_k\subset P_-$, $E(\hh_1)=E(\hh)-\mu(V_0)-\mu(V_k)$. As $\mu(V_0)=\mu(V_k)$ we have the conclusion.
\end{proof}

We construct better paths in $\hh$ with the help of the next two lemmas. Better paths are the analogue of extensions used in the last section.

\begin{lemma}\label{smain_lemma}
Let $A\subset P_+$, $B\subset P_0\cup P_+$, $T\subset X$, where $A\cap B=\es$. Then there exists $\vp:V\to W$ such that $graph(\vp)\subset\hh$, $V\subset A\cup B$, $W\subset (A\cup B\cup T)^c$ and:
\[\mu(V)\geqslant\frac{\mu(A)}{2n}-\frac{\mu(T)}2.\]
\end{lemma}
\begin{proof}
Let $\vp:V\to W$ be a maximal such piece. Then $N((A\cup B)\sm V)\subset (A\cup B\cup T\cup W)$. Consider the left and right counting measure of $([(A\cup B)\sm V]\times X)\cap\hh$ to get:
\[(n+1)\mu(A)+n\mu(B)-2n\mu(V)\leqslant (n-1)\mu(A)+n\mu(B)+2n\mu(T)+2n\mu(W).\]
As $\mu(V)=\mu(W)$, it follows that $\mu(A)\leqslant n\mu(T)+2n\mu(V)$, hence the conclusion.
\end{proof}

\begin{lemma}
Let $\psi_i:V_{i-1}\to W_{i}$ $i=1,\ldots,j$, $graph(\psi_i)\subset\hh$, be such that $W_0,W_1,\ldots, W_j$ are disjoint subsets, where $W_0=V_0$. Additionaly $W_i\subset P_0\cup P_+$ for $i>0$ and $V_i\subset W_0\cup\ldots\cup W_i$ for any $i\geqslant1$. Assume that $V_0\subset P_+$ and $W_j\not\subset (P_0\cup P_+)$. Then $\hh$ admits a better path of length smaller or equal to $j$.
\end{lemma}
\begin{proof}
Let $T_1=W_j\cap P_-$. By hypothesis $\mu(T_1)>0$. Then $\psi_j^{-1}(T_1)\subset V_{j-1}\subset W_0\cup\ldots\cup W_{j-1}$. It follows that there exists $i_1<j$ such that $\mu(\psi_j^{-1}(T_1)\cap W_{i_1})>0$.

Let $T_2=\psi_j^{-1}(T_1)\cap W_{i_1}$, so $T_2\subset W_{i_1}$. If $i_1=0$ then we are done, as $\psi_j$ restricted to $T_2$ is a better path of $\hh$ of lenght $1$. If $i_1>1$ then $\psi_{i_1}^{-1}(T_2)\subset V_{i_1-1}\subset W_0\cup\ldots\cup W_{i_1-1}$. Hence, there exists $i_2<i_1$ such that $\mu(T_3)>0$, where $T_3=\psi_{i_1}^{-1}(T_2)\cap W_{i_2}$.

Inductively define $T_{r+1}=\psi_{i_{r-1}}^{-1}(T_r)\cap W_{i_r}$ such that $\mu(T_{r+1})>0$. If $i_r>0$ then $T_{r+1}\subset W_{i_r}=\psi_{i_r}(V_{i_r-1})$, so $\psi_{i_r}^{-1}(T_{r+1})\subset V_{i_r-1}\subset W_0\cup\ldots\cup W_{i_r-1}$. Then there exists $i_{r+1}<i_r$ such that $T_{r+2}=\psi_{i_r}^{-1}(T_{r+1})\cap W_{i_{r+1}}$ and $\mu(T_{r+2})>0$. 

If $i_r=0$ then $T_{r+1}\subset V_0\subset P_+$ and this set $T_{r+1}$ transported with maps $\psi_{i_{r-1}},\ldots,\psi_{i_0}$ is a better path of $\hh$ of length $r$ (where $i_0=j$). As $0=i_r<i_{r-1}<\ldots<i_1<i_0=j$ it follows that $r\leqslant j$.
\end{proof}

We are now ready to prove the key proposition of this section. 

\begin{p}
Let $\hh$ be a division of $\gr$. Then there exists another division $\hh_1$ such that $E(\hh_1)<E(\hh)-(\frac{E(\hh)}{7n^3+E(\hh)})^2$.
\end{p}
\begin{proof}
Recall that $P_0=\{x:d_\hh(x)=n\}$, $P_-=\{x:d(x)<n\}$ and $P_+=\{x:d(x)>n\}$. Let $k=\lfloor\frac{7n^2}{\mu(P_+)}\rfloor$. 

Let $\{\vp_i^j:T_{i-1}^j\to T_{i}^j|i=1,\ldots,k(j)\}_j$ be a maximal collection of better paths of $\hh$ of length smaller or equal to $k$. This means that we require that the sets $(T_i^j)_{i,j}$ are disjoint. Additionally $k(j)\leqslant k$ for any $j$ and there is no extra better path $\{\vp_i\}_i$ of $\hh$ that can be added to the family without breaking at least one of these properties.

Let $T=\cup_jT_0^j$ and $T_e=\cup_{i,j}T_i^j$. Then $T\subset P_+$. As the maximal length of any better path is $k$, we have $\mu(T_e)\leqslant (k+1)\mu(T)$.

Using Lemma \ref{smain_lemma} for $(P_+\sm T_e)$, $\es$ and $T_e$, we deduce that there exists $\psi_1:W_0\to W_1$, $graph(\psi_0)\subset\hh$ such that $W_0\subset (P_+\sm T_e)$, $W_1\subset X\sm (P_+\cup T_e)$ and 
\[\mu(W_0)\geqslant\frac{\mu(P_+\sm T_e)}{2n}-\frac{\mu(T_e)}2.\]
If $W_1\not\subset P_0\cup P_+$ then we can restrict $\psi_1$ to a better path of length $1$, contradicting the maximality of the family $\{\vp_i^j\}_j$. It follows that $W_1\subset P_0\cup P_+$.

Now we apply Lemma \ref{smain_lemma} to the sets $W_0$, $W_1$ and $T_e$ to deduce the existence of $\psi_2:V_1\to W_2$, $graph(\psi_2)\subset\hh$ such that $V_1\subset W_0\cup W_1$, $W_2\subset X\sm(T_e\cup W_0\cup W_1)$ and:
\[\mu(W_2)\geqslant\frac{\mu(W_0)}{2n}-\frac{\mu(T_e)}{2}.\]
By the previous lemma and the maximality of the family $\{\vp_i^j\}_j$ we get $W_2\subset (P_0\cup P_+)$.

Inductively use Lemma \ref{smain_lemma} to the sets $W_0$, $W_1\cup\ldots\cup W_{j-1}$, and $T_e$ to deduce that there exists $\psi_j:V_{j-1}\to W_j$, $graph(\psi_j)\subset\hh$, such that $V_{j-1}\subset W_0\cup\ldots\cup W_{j-1}$, $W_j\subset X\sm(W_0\cup\ldots\cup W_{j-1}\cup T_e)$ and
\[\mu(W_j)\geqslant\frac{\mu(W_0)}{2n}-\frac{\mu(T_e)}{2}.\]
As long as $j\leqslant k$ we can use the previous lemma to deduce that $W_j\subset(P_0\cup P_+)$. Sets $W_0,\ldots,W_k$ are disjoint, so $1\geqslant\sum_j\mu(W_j)\geqslant(k+1)\mu(W_2)$. But:
\[\mu(W_2)\geqslant\frac{\mu(W_0)}{2n}-\frac{\mu(T_e)}{2}\geqslant\frac{\mu(P_+\sm T_e)}{4n^2}-\frac{\mu(T_e)}{4n}-\frac{\mu(T_e)}2=\frac{\mu(P_+)}{4n^2}-\frac{(2n^2+n+1)\mu(T_e)}{4n^2}.\]
As $\mu(T_e)\leqslant (k+1)\mu(T)$, we get:
\[4n^2\geqslant(k+1)\mu(P_+)-(k+1)^2(2n^2+n+1)\mu(T).\]
\begin{align*}
4n^2&\geqslant(k+1)\mu(P_+)-(k+1)^2(2n^2+n+1)\mu(T)>7n^2-(k+1)^2(2n^2+n+1)\mu(T)\\
\mu(T)&>\frac{3n^2}{(2n^2+n+1)(k+1)^2}>\frac1{(k+1)^2}>\big(\frac{E(\hh)}{7n^3+E(\hh)}\big)^2.
\end{align*}
We used $k\leqslant\frac{7n^2}{\mu(P_+)}\leqslant\frac{7n^3}{E(\hh)}$ for the last inequality. So, by replacing $graph(\vp_i^j)$ with $graph((\vp_i^j)^{-1})$ for each $i,j$, we get the division to satisfy the required inequality.
\end{proof}

\begin{te}
Let $\Psi$ be a symmetric DSE of multiplicity $2n$, $\gr$ its associated graph, and $\ve>0$. Then there exists $\hh$ a division of $\gr$ such that $E(\hh)<\ve$.
\end{te}
\begin{proof}
Procede as in the proof of Theorem \ref{pass_to_limit}.
\end{proof}

\begin{te}\label{th:symmDSE}
Let $\Psi$ be a symmetric DSE of multiplicity $2n$ and $\ve>0$. Then there exists $\Phi$ a DSE of multiplicity $n$ such that $d(\Psi,\mathcal{S}(\Phi))<\ve$.
\end{te}
\begin{proof}
Let $\gr$ be the associated graph of $\Psi$ and let $\hh$ be a division of $\gr$ such that $E(\hh)<\ve/4$. We can assume that there is no map $\vp:A\to B$ with $A\subset P_+$ and $B\subset P_-$, i.e. there is no better path of length $1$. If there is such a path, just replace $graph(\vp)$ by $graph(\vp^{-1})$.

We first construct $\hh_1\subset\hh$ such that $\nu(\hh\sm\hh_1)=E(\hh)$ and the in-degree and out-degree of each vertex in $\hh_1$ is less than $n$. Formally this can be written as $d_{\hh_1}(x)\leqslant n$ and $d_{flip(\hh_1)}(x)\leqslant n$ for each $x\in X$.

For this construction, choose a sequence of maps $\vp_i:A_i\to B_i$, with $A_i\subset P_+$ and consequently $B_i\subset (P_-)^c$, such that $graph(\vp_i)\subset\hh$ and $\sum_i\chi_{A_i}(x)=d_\hh(x)-n$ for each $x\in P_+$. As $\int_{P_+}(d_\hh(x)-n)d\mu(x)=\frac12E(\hh)$ it follows that $\sum_i\mu(A_i)=\frac12E(\hh)$.

Symmetrically, choose a sequence of maps $\psi_i:C_i\to D_i$, with $D_i\subset P_-$, so that $C_i\subset P_+^c$, such that $graph(\psi_i)\subset\hh$ and $\sum_i\chi_{D_i}(x)=n-d_\hh(x)$ for each $x\in P_-$. Then $\sum_i\mu(D_i)=\frac12E(\hh)$.

Now $\hh_1$ can be defined as $\hh\sm(\cup_igraph(\vp_i))\sm(\cup_igraph(\psi_i))$. Choose an arbitrary sequence of measure preserving partial morphisms $\theta_i:S_i\to T_i$ such that $\sum_i\chi_{S_i}=\sum_i\chi_{C_i}$ and $\sum_i\chi_{T_i}=\sum_i\chi_{B_i}$. Also, choose measure preserving partial morphisms $\delta_i:V_i\to W_i$ such that $\sum_i\chi_{V_i}=\sum_i\chi_{D_i}$ and $\sum_i\chi_{W_i}=\sum_i\chi_{A_i}$. Define $\hh_2=\hh_1\cup(\cup_igraph(\theta_i)\cup(\cup_igraph(\delta_i))$. Then $\nu(\hh_2\Delta\hh)=2E(\hh)$ and $d_{\hh_2}(x)=d_{flip(\hh_2)}(x)=n$ for all $x\in X$. This implies that $\nu(\gr\Delta(\hh_2\cup flip(\hh_2))\leqslant 4E(\hh)<\ve$. Using Theorem \ref{Kechris} we can transform $\hh_2$ into a DSE of multiplicity $n$ to finish the proof.
\end{proof}

\subsection{Applications to Borel graphs}

Putting together Theorems \ref{th:decDSE} and \ref{th:symmDSE} we get that any $2n$-regular graph almost contains a measure-preserving automorphism.

\begin{p}
Let $\gr$ be a measure preserving, $2n$-regular graph and $\ve>0$. Then there exists $\theta:A\to B$ such that $graph(\theta)\subset\gr$ and $\mu(A)>1-\ve$.
\end{p}

\section{Sofic doubly stochastic elements}

\subsection{Preliminaries} We quickly recall the notion of sofic equivalence relation. For a more detailed introduction to the subject, the reader can consult \cite{Pa1}. Roughly speaking, an equivalence relation $E$ is \emph{sofic} if infinite matrices over $E$ can be locally approximated by finite matrices. We make this definition more concrete by introducing the algebra $M_f(E)$ (of infinite matrices over $E$) and by discussing ultraproducts of matrix algebras. We begin with the latter. 

Fix a sequence $\{m_k\}_k\subset\nz$ such that $\lim_km_k=\infty$ and $\omega$ a free ultrafilter on $\nz$. Let $M_m$ be the *-algebra of matrices in dimension $m$, endowed with the trace: $Tr(a)=\frac1m\sum_{i=1}^ma(i,i)$. Note that $Tr(Id_m)=1$, independent of $m\in\nz$. Recall that $||x||_2=\sqrt{Tr(x^*x)}$.

\begin{de}
The \emph{matrix ultraproduct} is defined as $\Pi_{k\to\omega}M_{m_k}=l^\infty(\nz,M_{m_k})/\nr_\omega$, where $l^\infty(\nz,M_{m_k})=\{(x_k)_k\in\Pi_kM_{m_k}:sup_k||x_k||<\infty\}$ is the set of bounded sequences of matrices w.r.t the operator norm, and $\nr_\omega=\{(x_k)_k\in l^\infty(\nz,M_{m_k}):\lim_{k\to\omega}||x_k||_2=0\}$.
\end{de}

On the ultraproduct $\Pi_{k\to\omega}M_{m_k}$ define the trace $Tr(\Pi_{k\to\omega}x_k)=\lim_{k\to\omega}Tr(x_k)$, where $\Pi_{k\to\omega}x_k$ is the generic element in $\Pi_{k\to\omega}M_{m_k}$ with $x_k\in M_{m_k}$. The sets $D_m\subset M_m$ and $P_m\subset M_m$ are the \emph{subalgebra of diagonal matrices} and \emph{subgroup of permutation matrices} respectively. 

We now construct the algebra $M_f(E)$. As before $E\subset X^2$ is a countable measure preserving equivalence relation on a standard probability space $(X,\bb,\mu)$. The next definitions are from \cite{Fe-Mo}.

\begin{de}\label{finite function}
A measurable function $f:E\to\cz$ is called \emph{finite} if $f$ is bounded and there exists $n\in\nz$ such that $|\{z:f(x,z)\neq 0\}|\leqslant n$ and $|\{z:f(z,y)\neq 0\}|\leqslant n$ for $\mu$-almost any $x,y$.
\end{de}

\begin{p}
The set $M_f(E)=\{f:E\to\cz:f\mbox{ finite}\}$ is a *-algebra endowed with a trace. The operations are defined as follows:
\begin{align*}
(f+g)(x,y)=&f(x,y)+g(x,y);&
(f\cdot g)(x,y)=&\sum_zf(x,z)g(z,y);\\
f^*(x,y)=&\overline{f(y,x)};&
Tr(f)=&\int_xf(x,x)d\mu(x).
\end{align*}
\end{p}

We need the equivalent in $M_f(E)$ of a diagonal and a permutation matrix. The algebra $L^\infty(X,\mu)$ canonically embeds in $(M_f(E),Tr)$ by $L^\infty(X,\mu)\ni a\to\tilde a\in M_f(E)$, where $\tilde a(x,y)=a(x)\delta_x^y$.
For $\vp\in[E]$ define $u_\vp=\chi(graph(\vp_i^{-1}))\in M_f(E)$. Then $u_\vp$ is a unitary in $M_f(E)$ and $[E]\ni\vp\to u_\vp\in M_f(E)$ is a group morphism.

We can now state de definition.

\begin{de}
The equivalence relation $E$ is \emph{sofic} if there exists a trace preserving embedding $\theta:M_f(E)\to\Pi_{k\to\omega}M_{m_k}$ such that $\theta(\tilde a)\in\Pi_{k\to\omega}D_{m_k}$ and $\theta(u_\vp)\in\Pi_{k\to\omega}P_{m_k}$ for each $a\in L^\infty(X,\mu)$ and $\vp\in [E]$.
\end{de}

Let $PP_m=\pp(D_m)\cdot P_m$, i.e. $PP_m$ is the set of permutations cut with a diagonal projection. It can be deduced, from the definition of sofic equivalence relation, that if $v\in[[E]]$ and $\theta:M_f(E)\to\Pi_{k\to\omega}M_{m_k}$ is a sofic embedding, then $\theta(v)\in\Pi_{k\to\omega}PP_{m_k}$.

\subsection{Sofic DSE} A doubly stochastic element and its associated matrix is the same information. Moreover, the associated matrix of a DSE is a finite function in the sense of Definition \ref{finite function}. Thus, by approximating the associated matrix of a DSE with finite matrices, we hope to derive some conclusions form the classic Birkhoff - von Neumann theorem.

\begin{de}
A DSE $\Phi=\{\vp_i:i\}$ is called \emph{sofic} if the orbit equivalence relation generated by the maps $(\vp_i)_i$ is sofic.
\end{de}

Our goal is to prove that the associated matrix of a sofic DSE can be approximated by doubly stochastic matrices. We use the following lemma. A proof can be found in \cite{Ar-Pa}, Lemma 6.3.

\begin{lemma}
Let $\{c_j\}_j\subseteq\pp(\Pi_{k\to\omega}D_{m_k})$ be a sequence of projections such that $\sum_jc^j=\id$. Then there exist projections $c_j^k\in \pp(D_{m_k})$ such that $c_j=\Pi_{k\to\omega}c_j^k$ and $\sum_jc_j^k=\id_{m_k}$ for each $k\in\nz$.
\end{lemma}

\begin{p}
Let $\Phi=\{\vp_i:i\}$ be a sofic DSE of multiplicity $n$, $E$ the equivalence relation generated and let $\theta:M_f(E)\to\Pi_{k\to\omega}D_{m_k}$ be a sofic embedding. Then there exists $x_k\in B_{m_k}^n$ such that $\theta(M(\Phi))=\Pi_{k\to\omega}x_k$.
\end{p}
\begin{proof}
Consider $\vp_i:A_i\to B_i$, $i=1,\ldots,r$. Define $v_i=\chi(graph(\vp_i^{-1}))\in M_f(E)$. Then, inside $M_f(E)$, we have $v_i^*v_i=\tilde\chi(A_i)$ and $v_iv_i^*=\tilde\chi(B_i)$.

Let $\{C_j\}_{j=1}^s$ be the partition of $X$ generated by sets $(A_i)_i$. Then, for each $i=1,\ldots,r$, $A_i=\sqcup_{j\in S_i}C_j$ for some $S_i\subset\{1,\ldots,s\}$. Because $\Phi$ is a DSE of multiplicity $n$, each $j=1,\ldots,s$ belongs to exactly $n$ sets from the colection $S_1,\ldots,S_r$. Similarly let $\{D_j\}_j$ be the partition generated by sets $\{B_i\}_i$ such that $B_i=\sqcup_{j\in T_i}D_j$. Routine partial isometry computations show that $\tilde\chi(B_i)\cdot v_i\cdot\tilde\chi(A_i)=v_i$.

Let $c_j=\tilde\chi(C_j)\in M_f(E)$. Then $\sum_{j=1}^sc_j=\id$. Using the previous lemma, we find $c_j^k\in \pp(D_{m_k})$ such that $\theta(c_j)=\Pi_{k\to\omega}c_j^k$ and $\sum_jc_j^k=\id_{m_k}$ for each $k\in\nz$. We construct $d_j\in M_f(E)$ and $d_j^k\in \pp(D_{m_k})$ in a similar way. Notice that $\sum_i\sum_{j\in S_i}c_j^k=n\sum_jc_j^k=n\cdot\id$.

Construct $v_i^k\in PP_{m_k}$, $\theta(v_i)=\Pi_{k\to\omega}v_i^k$. Define $w_i^k=(\sum_{j\in T_i}d_j^k)v_i^k(\sum_{j\in S_i}c_j^k)$. Then $\Pi_{k\to\omega}w_i^k=(\sum_{j\in T_i}\theta(d_j))\cdot\theta(v_i)\cdot(\sum_ {j\in S_i}\theta(c_j))=\theta(\tilde\chi(B_i))\cdot\theta(v_i)\cdot\theta(\tilde\chi(A_i))=\theta(v_i)$. Additionally $(w_i^k)^*w_i^k\leqslant\sum_{j\in S_i}c_j^k$ and $w_i^k(w_i^k)^*\leqslant \sum_{j\in T_i}d_j^k$.  

Let $y_k=\sum_i w_i^k\in M_{m_k}(\nz)$. Then $\Pi_{k\to\omega}y_k=\theta(M(\Phi))$. As $(w_i^k)^*w_i^k\leqslant\sum_{j\in S_i}c_j^k$ and $\sum_i\sum_{j\in S_i}c_j^k=n\cdot\id$, it follows that the sum of all entries in $y_k$ on each column is less than $n$. The same statement is valid for the sum of entries on each row. In order to finish the proof, we show that it is possible to increase some entries in $y_k$ without changing the value of $\Pi_{k\to\omega}y_k$.

Let $t_k$ be the sum of entries of $y_k$ divided by $m_k$. Clearly $t_k\leqslant n$. As $w_i^k\in PP_{m_k}$ the sum of entries divided by $m_k$ for this matrix is $Tr((w_i^k)^*w_i^k)$. Then $t_k=\sum_iTr((w_i^k)^*w_i^k)$ and $\lim_{k\to\omega}t_k=\sum_iTr(\theta(v_i)^*\theta(v_i))=\sum_i\mu(A_i)=n$.

Construct a matrix $z_k\in M_{m_k}(\nz)$ such that $y_k+z_k\in B_{m_k}^n$. The sum of entries in $z_k$ divided by $m_k$ is $n-t_k$ and each entry is smaller than $n$. Then $||z_k||_2^2\leqslant n^2(n-t_k)\to_{k\to\omega}0$. It follows that $\Pi_{k\to\omega}z_k=0$ so $x_k=y_k+z_k$ are the required matrices.
\end{proof}

\begin{te}
Let $\Phi=\{\vp_i:i\}$ be a sofic DSE of multiplicity $n$, $E$ the equivalence relation generated and let $\theta:M_f(E)\to\Pi_{k\to\omega}D_{m_k}$ be a sofic embedding. Then there exists $p_1,\ldots,p_n\in\Pi_{k\to\omega}P_{m_k}$ such that $\theta(M(\Phi))=\sum_{i=1}^np_i$.
\end{te}
\begin{proof}
This follows from the previous proposition and the classic Birkhoff - von Neumann theorem.
\end{proof}

If $p_1,\ldots,p_n$ from the theorem are elements in $\theta(M_f(E))$ then $\Phi$ is decomposable.
Otherwise, consider the von Neumann algebra $A$ generated by $\theta(M_f(E))$ and $p_1,\ldots,p_n$. Then $A\cap\Pi_{k\to\omega}D_{m_k}$ is abelian, so that $(A\cap\Pi_{k\to\omega}D_{m_k})\sim L^\infty(Y)$. Now $p_1,\ldots,p_n$ act on $L^\infty(Y)\supset\theta(L^\infty(X))$. It follows that the original DSE $\Phi$ can be amplified to a decomposable DSE on the space $Y$.

\subsection{Example} We present here a decomposable amplification of Example \ref{counterexample}. This example is composed of partial isomorphisms $\Phi=\{\vp_1,\vp_2,\psi_n^1,\psi_n^2\}$ of the unit interval.

Let $I$ denote the unit interval. The new DSE, $\tilde\Phi$, is constructed on $I\times\{1,2\}$, and it is composed of the maps:
\begin{align*}
\tilde\vp_{1,1}(x,1)&=(\vp_1(x),1);&\ \tilde\psi_n^{1,1}(x,1)&=(\psi_n^1(x),1)\\
\tilde\vp_{1,2}(x,2)&=(\vp_1(x),2);&\ \tilde\psi_n^{1,2}(x,2)&=(\psi_n^1(x),2)\\
\tilde\vp_{2,1}(x,1)&=(\vp_2(x),2);&\ \tilde\psi_n^{2,1}(x,1)&=(\psi_n^2(x),2)\\
\tilde\vp_{2,2}(x,2)&=(\vp_2(x),1);&\ \tilde\psi_n^{2,2}(x,2)&=(\psi_n^2(x),1)\\
\end{align*}
Define on $I\times\{1,2\}$ the equivalence relation $(x,1)\sim(x,2)$ for any $x\in I$. Then, under this equivalence relation, $\tilde\Phi$ collapses to the old $\Phi$. Moreover $\tilde\Phi$ can be decomposes in two isomorphisms of $I\times\{1,2\}$ by pasting pieces $\{\tilde\vp_{1,2},\tilde\vp_{2,1},\tilde\psi_n^{1,1},\tilde\psi_n^{2,2}\}$ into one isomorphism and pieces $\{\tilde\vp_{1,1},\tilde\vp_{2,2},\tilde\psi_n^{1,2},\tilde\psi_n^{2,1}\}$ into the other.

\section{Application to Hecke operators}

Consider a countable discrete group $G$ acting ergodicaly and a.e. free, by measure preserving transformations on an infinite measure space $(X,\mu)$, with $\sigma$-finite measure $\mu$. Let  $\Gamma \subseteq G$ be an almost normal subgroup. By definition, a subgroup is \emph{almost normal} if for all $g\in G$ the group $\Gamma_g= \Gamma\cap g\Gamma g^{-1}$ has finite index in $\Gamma$. Assume that the restriction of the action $G\curvearrowright (X,\mu)$ to $\Gamma$ admits a finite measure, fundamental domain $F\subseteq X $.  We consider the countable,  measurable equivalence relation $\mathcal R_G$ on $X$ induced by the orbits of $G$, and let $\mathcal R_G| F$ be its restriction to $F$ (thus two points in $F$ are equivalent if and only if they are on the same orbit of $G$).

%\begin{defn}\label{pi}
% Let $G$ be a discrete group acting by measure preserving
%transformations, almost everywhere freely on  $\cX$. Assume that $\Gamma$ is an almost normal subgroup, having a  fundamental domain $F\subseteq \cX$, of finite
%measure.

%Let as above $\cR_G| F$ be the countable equivalence relation on $F$, defined by the requirement
%that $x\sim y$ if $Gx = Gy$.

For $g$ in $G$, we introduce  
$\dot{\wideparen{\Gamma g}}$, a function mapping  $F$ with values in $F$, constructed  as follows:
Let $x$ be an element  in $F$. Since $F$ is a fundamental domain, there exists
a unique $\gamma_1\in \Gamma$ and $x_1$ in $F$ such that $gx=\gamma_1x_1$. Then we define:
\begin{equation}\label{gammag}
\dot{\wideparen{\Gamma g}} (x):=x_1=\gamma_1^{-1}gx.
\end{equation}
\noindent Clearly, the function $\dot{\wideparen{\Gamma g}}$
depends only on the left $\Gamma$-coset $\Gamma g$, for all  $g\in G$. 

Then $\cR_G|_F$ is generated by the transformations $\dot{\wideparen{\Gamma g}}$, $g$ running 
through a system of representatives for left cosets of $\Gamma$. Indeed, the above definition implies that $x\sim y$ with respect to $\cR_G|_F$  if  and only if there exists $g\in G$ such that
$\dot{\wideparen{\Gamma g}} x=y$.

Let $\Gamma g \Gamma$ be the double coset associated  to $g$. Assume that $(r_j)_{j=1,\dots,[\Gamma:\Gamma_g]}$  are a system of left  coset representatives for $\Gamma_{g^{-1}}$ in $\Gamma$. Thus, $\Gamma$ is the disjoint reunion of $\Gamma_{g^{-1}}\cdot r_j$,  $j=1,\dots,[\Gamma:\Gamma_g]$.    This is equivalent to the fact  that  $\Gamma g \Gamma$ is  a finite reunion of right cosets of $\Gamma$:
$$\Gamma g \Gamma=\mathop \bigcup_{j=1}^{[\Gamma:\Gamma_g]} \Gamma g r_j.$$

Let $T_{\Gamma g \Gamma}:L^2(F,\mu)\to L^2(F,\mu)$ be the Hecke operator associated to the double coset
$\Gamma g \Gamma$ %where $r_j$ are the coset representatives for $\Gamma_g\subseteq \Gamma$ 
(see e.g. \cite {Krieg}).
Then
\begin{equation}\label{hecke}
T_{\Gamma g \Gamma}=\mathop\sum_{j=1}^{[\Gamma:\Gamma_g]} \dot{\wideparen{\Gamma gr_j}}. 
\end{equation}
 Assume that  $$[\Gamma:\Gamma_g]=[\Gamma:\Gamma_{g^{-1}}]$$for all $g\in G$. 
Recall (\cite {Krieg}) that in this case $$\Gamma g\Gamma=\Gamma g^{-1}\Gamma, \ g \in G.$$

It is  proven in \cite {Ra} that there exists a finite measurable partition of $F$, consisting of  sets $(A^g_{r_j})_{{j=1},\dots, {[\Gamma:\Gamma_g]}} $, such that the restriction of $\dot{\wideparen{\Gamma g}} $  to each of the sets $A^g_{r_j}$, ${j=1,\dots, [\Gamma:\Gamma_g]}$, is injective. These restrictions do not necessary have disjoint images.

 Also, in the paper cited above, it is proved that 
the inverse of  the transformation $\dot{\wideparen{\Gamma g}}$, when restricted to a domain of injectivity as above,
is the restriction to an analogous  injectivity domain of the function $\dot{\wideparen{\Gamma h}}$, for a left $\Gamma$-coset  $\Gamma h$,  contained in $\Gamma g\Gamma$. We also assume that none of the transformations above is equal to its own inverse. This is a restriction imposed on the Hecke algebra of double cosets and it amounts to the fact that for all $g\in G\setminus \{e\}$, the cosets $\Gamma g$ and $\Gamma g^{-1}$ are distinct. This condition is verified in the example we are considering below (\cite{Krieg}).

 The  formula (\ref{hecke}) implies that, although  the  function $\dot{\wideparen{\Gamma g}}$ is not injective,  the cardinality of the   set 
 $$\{ \dot{\wideparen{\Gamma g r_j}}(f) \mid j=1,\dots, [\Gamma:\Gamma_g] \},$$
 for $f$ in $F$, is constantly equal to $[\Gamma:\Gamma_g]$. Indeed, the points enumerated in the set above are the Hecke points corresponding to $f\in F$. Because $G$ acts freely a. e. they are a.e.
distinct (\cite{COU}). The same is true for the set of preimages.
% number of preimages
%of each point in the image
%is equal to $[\Gamma:\Gamma_g]$, where %In addition, if $\Gamma gs_i$ are the left $\Gamma$-cosets contained in   $\Gamma g \Gamma$,
%hen every point $x$ in $F$ will show up exactly $[\Gamma:\Gamma_g]$-times in the
%reunion of the images of the maps $\dot{\wideparen{\Gamma g}}s_i$. 
%The same
%is true for preimages.
% (with $[\Gamma:\Gamma_{g^{-1}}]$ instead of $[\Gamma:\Gamma_g]$).
%\end{defn}

Consider  the finite set of partial transformations of $F$, denoted by  $\mathcal D_{\Gamma g\Gamma}$, consisting of the restrictions of the functions   $\dot{\wideparen{\Gamma gr_j}} $ to domains  of injectivity as above. Then $\mathcal D_{\Gamma g\Gamma}$
is  a symmetric DSE of order $[\Gamma:\Gamma_g]$. 
Because of formula (\ref{hecke}) we obtain  that 
\begin{equation}
T_{\Gamma g \Gamma}= \mathop\sum_{t\in \mathcal D_{\Gamma g\Gamma}} \ t.
\end{equation}

In the case $G = \PGL_2(\Z[\frac1p])$, $\Gamma=\PSL_2(\Z)$,  $p\geq 3$, a prime number,  the relation $\mathcal R_G| F$ is the equivalence relation associated to a free, measure preserving  action, on $F$, of a free group with $(p+1)/2$ generators. Indeed let 
$$\sigma_{p^n}=\left(
\begin{matrix}
p^n &0\\0&1
\end{matrix}\right), n \in \mathbb N.$$
\noindent 
 For cosets $\Gamma g_1, \Gamma g_2,\dots, \Gamma g_n$ contained in $\Gamma\sigma_p\Gamma$,  any relation of the form $$\dot{\wideparen{\Gamma g_1}}\;\dot{\wideparen{\Gamma g_2}}\ldots\dot{\wideparen{\Gamma g_n}}f=f,$$
for  $f\in F$, is possible if and only if each  the factors  $\dot{\wideparen{\Gamma g_i}}$
is canceled by its  inverse. Hence, since $[\Gamma: \Gamma_{\sigma_p}]= \frac{p+1}{2}$, the equivalence relation
$\cR_G|F$ is treeable, of cost $\frac{p+1}{2}$. The graphing of this equivalence relation consist of the partial  transformations in the set $\mathcal D_{\Gamma \sigma_p\Gamma}$.   %(its generators and inverses
%being $\dot{\wideparen{\Gamma g}}$, with $\Gamma g \subseteq \Gamma \sigma_p \Gamma$).
%Because of formula (\ref{hecke}$ we have that 

By Hjorth theorem ([Hj]), there exists a free group factor $F_{\frac{p+1}{2}}$ acting
freely on $F$, whose orbits are the equivalence relation in $\cR_G|F$.

Because of the Theorems \ref{th:decDSE} and \ref{th:symmDSE} on symmetric DSE, we can arrange that that generators of $F_{\frac{p+1}{2}}$ are built, for every $\varepsilon >0$,  up to a subset $ F_\varepsilon \subseteq F$ of  measure less than $\varepsilon$, 
from restrictions to smaller domains of the transformations of $\dot{\wideparen{\Gamma g}}$, $\Gamma g \subseteq \Gamma \sigma_p \Gamma$, glued together into injective  transformations defined  on $L^2(F\setminus F_\varepsilon, \nu)$.    All the elements in $\mathcal D_{\Gamma \sigma_p\Gamma}$ are used exactly once in this process (up to a subset of measure less that $\varepsilon$.
 %Moreover, in this construction, the set  $F\setminus F_\varepsilon , \varepsilon >0$, is  a. e. invariant to the action of $F_{\frac{p+1}{2}}$.
 
Recall that   the radial elements in group algebra of the group $F_{\frac{p+1}{2}}$ are  the selfadjoint elements $\chi_n\in \mathbb C(F_{\frac{p+1}{2}})$ equal to the sum of words, in 
the generators of $F_{\frac{p+1}{2}}$,   of length $n$, $n\in \N$. The above argument shows that the image, through the Koopman unitary representation of   $\chi_n$, restricted to $L^2(F\setminus F_\varepsilon,\nu )$ with values in $L^2(F,\nu )$, coincides with the restriction of the Hecke operator $T_{\Gamma \sigma_{p}\Gamma }\mid_{L^2(F\setminus F_\varepsilon,\nu)}$.

 %In particular any spectral gap in the spectrum of the Hecke operators  $T_{\Gamma\sigma_p\Gamma}$ will appear from  the limit, as $\varepsilon$ tends to $0$, of the spectral gap
% of the image, through the Koopman unitary representation of $F_{\frac{p+1}{2}}$ on 
%$L^2(F\setminus F_\varepsilon,\nu)$, of the radial algebra element $\chi_1\in \mathbb C(F_{\frac{p+1}{2}})$.

Consequently the spectral gap behavior of the Hecke operators associated to the action of $G$ on $X$ is similar to the spectral gap phenomena considered in the paper
\cite{LPS}.

\end{document}